\documentclass[11pt]{amsart}

\usepackage[T1]{fontenc}
\usepackage[utf8]{inputenc}
\usepackage{lmodern}
\usepackage{microtype}
\usepackage{amsmath,amssymb,amsthm,mathtools}
\usepackage{enumitem}
\usepackage{xcolor}
\usepackage[colorlinks=true,linkcolor=blue!55!black,citecolor=blue!55!black,urlcolor=blue!55!black]{hyperref}
\usepackage[capitalise,noabbrev]{cleveref}

\newtheorem{theorem}{Theorem}[section]

\newtheorem{corollary}[theorem]{Corollary}

\newtheorem{problem}[theorem]{Problem}
\theoremstyle{definition}

\theoremstyle{remark}
\newtheorem{remark}[theorem]{Remark}

\newcommand{\R}{\mathbb R}
\newcommand{\eps}{\varepsilon}
\newcommand{\Ad}{\operatorname{Ad}}
\newcommand{\ad}{\operatorname{ad}}
\newcommand{\Span}{\operatorname{span}}
\newcommand{\rank}{\operatorname{rank}}
\newcommand{\SO}{\mathrm{SO}}
\newcommand{\SU}{\mathrm{SU}}
\newcommand{\Sp}{\mathrm{Sp}}

\title[Adjoint reductions of tangent spheres]{Adjoint reductions of tangent bundles of spheres}
\author{Leonardo Mart\'inez-Sandoval}
\address{Facultad de Ciencias, Universidad Nacional Aut\'onoma de M\'exico (UNAM), Mexico}
\email{leomtz@ciencias.unam.mx}
\urladdr{https://orcid.org/0000-0002-5104-9635}
\dedicatory{Dedicated to Luis Montejano on the occasion of his 75th birthday.}
\date{\today}

\subjclass[2020]{Primary 55R25, 57R15; Secondary 22E46, 46B20, 52A20}
\keywords{adjoint representation, reduction of structure group, tangent bundle of a sphere, $G$-structures on spheres, Stiefel manifold, vector fields on spheres, exceptional Lie group $E_7$, Banach isometric subspace problem}

\begin{document}

\begin{abstract}
We study reductions of the tangent bundle of $S^{\dim G}$ through the adjoint
representation of a compact connected Lie group $G$.  If $d=\dim G$ and
$r=\rank G$, we show that the adjoint map $\Ad\colon G\longrightarrow \SO(d)$ 
is homotopic, as an ordinary map, to one with values in $\SO(d-r+1)$.  It follows that every vector bundle over a sphere associated to a principal $G$-bundle via the adjoint representation admits $r-1$ linearly independent sections.
Combined with Adams's theorem on vector fields on spheres, this gives the necessary
condition $r\le \rho(d+1)$
for an adjoint reduction of $TS^d$.

In particular, no such reduction exists for a
non-trivial compact connected group, simple or not, of dimension $d>1$ with
$d\not\equiv3\pmod4$.  As an application, this excludes the adjoint $E_7$-reduction
of $TS^{133}$ that is the exceptional branch in the structure-group argument of
Bor, Hern\'andez-Lamoneda, Jim\'enez-Desantiago and Montejano for Banach's isometric
subspace problem.
\end{abstract}

\maketitle

\section{Introduction}

Reductions of the structure group of tangent bundles of spheres are a classical
problem in topology.  Leonard and Ozaki obtained broad nonexistence results for
proper reductions, while \v{C}adek and Crabb later classified reductions through
homomorphisms between the classical groups $\SO$, $\SU$ and $\Sp$
\cite{Leonard1971,Ozaki1991,CadekCrabb2006}.  The general problem is an unstable
one and can require detailed information about both the homomorphism and the
homotopy type of the groups involved.

Here we isolate a narrower class of reductions.  Let $G$ be a $d$-dimensional compact connected
Lie group of rank $r$. Consider
its adjoint representation
\[
  \Ad\colon G\longrightarrow\SO(\mathfrak g)\cong\SO(d).
\]

We ask when the tangent bundle of $S^d$ can reduce through this map.
Our main result is the following.

\medskip
\noindent\textbf{Theorem~\ref{thm:rank-obstruction}.}
\emph{Let $G$ be a $d$-dimensional compact connected Lie group of rank $r$.  If
$TS^d$ admits a reduction through
\[
  \Ad\colon G\longrightarrow\SO(d),
\]
then
\[
  r\le \rho(d+1),
\]
where $\rho$ is the Radon--Hurwitz function.}
\medskip

The numerical condition has an immediate uniform consequence.

\medskip
\noindent\textbf{Corollary~\ref{cor:mod-four}.}
\emph{Let $G$ be a nontrivial compact connected Lie group of dimension $d>1$.  If $d\not\equiv3\pmod4$, 
then $TS^d$ admits no reduction through the adjoint representation of $G$.}
\medskip

For many individual simple groups, stronger non-reduction statements follow from
the classical theory of $G$-structures on spheres.  The point here is different:
the obstruction above is uniform, uses only the intrinsic rank and dimension of
$G$, and applies without assuming that $G$ is simple or classical.

One motivation for the question comes from Banach's isometric subspace problem. Banach posed his isometric subspace problem in 1932 \cite{Banach1932}.
The real case $n=2$ was settled by Auerbach, Mazur and Ulam
\cite{AuerbachMazurUlam1935}, while Dvoretzky proved the conjecture for
infinite-dimensional real spaces \cite{Dvoretzky1959}.  Gromov subsequently
proved it for even $n$ and, for real spaces, for odd $n$ whenever
$\dim V\ge n+2$ \cite{Gromov1967}.  In the complex finite-dimensional setting,
Bracho and Montejano proved the conjecture for $n\equiv1\pmod4$
\cite{BrachoMontejano2021}.  The real case $n=3$ was settled by Ivanov,
Mamaev and Nordskova \cite{IvanovMamaevNordskova2023}.  For real
$n=4k+1\ge5$, Bor, Hern\'andez-Lamoneda, Jim\'enez-Desantiago and
Montejano proved the conjecture with the possible exception of $n=133$
\cite{BorEtAl2021}.\footnote{Recent preprints have also claimed complete
solutions. Zhang posted such a claim in 2025 \cite{Zhang2025}. While this
manuscript was being prepared, Lu and Yang posted another
\cite{LuYang2026}. We have not undertaken a detailed verification of these
claims and make no assertion here about their correctness; the latter
appeared too recently for a detailed comparison.} Their
structure-group analysis isolates precisely one exceptional irreducible branch:
a reduction of $TS^{133}$ through the adjoint group of type $E_7$
\cite{BorEtAl2021}.  Since $\dim E_7=133\equiv1\pmod4$,
Corollary~\ref{cor:mod-four} excludes that branch.

The proof of the obstruction is topologically economical once the adjoint structure is exploited. A fixed-frame compression argument, based on the connectivity of Stiefel manifolds and the homotopy lifting property, shows that the adjoint map can be deformed into a smaller orthogonal group in a way controlled by the rank of G. Via clutching, this compression forces trivial line summands in the corresponding adjoint bundles over spheres. Adams's theorem on vector fields then gives the obstruction. Thus, aside from Adams's deep theorem, the argument uses only standard facts about bundles, Stiefel manifolds, and clutching.

The paper is organized as follows.  \Cref{sec:compression} develops the fixed-frame
compression criterion, and \cref{sec:adjoint} specializes it to adjoint
representations.  \Cref{sec:bundles-spheres} passes to bundles over spheres,
proves \cref{thm:rank-obstruction}, and compares the resulting obstruction with the
classical $G$-structure literature.  \Cref{sec:banach} records the
$133$-dimensional application, and \cref{sec:further} ends with the resulting
classification question for adjoint reductions.

\section{Fixed-frame compression}\label{sec:compression}

Let $V_{N,q}$ be the real Stiefel manifold of ordered orthonormal $q$-frames in
$\R^N$.  For the standard frame $\mathbf e=(e_1,\ldots,e_q)$, remembering the
first $q$ columns gives the fiber bundle
\begin{equation}\label{eq:stiefel-fibration}
  \SO(N-q)\hookrightarrow \SO(N)
  \xrightarrow{\ p_q\ }V_{N,q},
  \qquad
  p_q(A)=(Ae_1,\ldots,Ae_q).
\end{equation}
Its fiber over $\mathbf e$ is the subgroup fixing $e_1,\ldots,e_q$ pointwise.  We
use the standard fact that
\begin{equation}\label{eq:stiefel-connectivity}
  V_{N,q}\text{ is }(N-q-1)\text{-connected};
\end{equation}
see, for example, \cite[Chapter~8, Theorem 5.1]{Husemoller1994}.

A continuous map $a\colon X\to\SO(N)$ will be called \emph{$q$-compressible} if it
is homotopic to a map whose image lies in the standard subgroup
$\SO(N-q)\subset\SO(N)$.

\begin{theorem}\label{thm:fixed-frame}
Let $G$ be a compact connected Lie group and let
\[
  \rho\colon G\longrightarrow\SO(N)
\]
be an orthogonal representation.  Suppose that a closed subgroup $K\subset G$ acts
trivially on a $q$-dimensional subspace $W\subset\R^N$, where $1\le q<N$, and that
\begin{equation}\label{eq:dimension-condition}
  \dim(G/K)\le N-q-1.
\end{equation}
Then $\rho$ is $q$-compressible.
\end{theorem}

\begin{proof}
Choose an ordered orthonormal frame $\mathbf w=(w_1,\ldots,w_q)$ in $W$.  Define
\[
  F_{\mathbf w}\colon G\longrightarrow V_{N,q},
  \qquad
  F_{\mathbf w}(g)=\rho(g)\mathbf w.
\]
For $k\in K$,
\[
  F_{\mathbf w}(gk)
  =\rho(g)\rho(k)\mathbf w
  =\rho(g)\mathbf w
  =F_{\mathbf w}(g),
\]
because $K$ fixes $W$ pointwise.  Hence $F_{\mathbf w}$ factors as
\[
  G\longrightarrow G/K
  \xrightarrow{\ \bar\rho_{\mathbf w}\ }V_{N,q}.
\]
The homogeneous space $G/K$ is a finite CW complex of dimension $\dim(G/K)$.
By \eqref{eq:dimension-condition} and \eqref{eq:stiefel-connectivity},
$\bar\rho_{\mathbf w}$ is null-homotopic, and hence so is
$F_{\mathbf w}$.

We now compare the fixed frame $\mathbf w$ with the standard frame used in
\eqref{eq:stiefel-fibration}.  Since $q<N$, the Stiefel manifold $V_{N,q}$ is
connected.  Choose a path of frames $\mathbf a_t$ from
$\mathbf a_0=\mathbf e$ to $\mathbf a_1=\mathbf w$.  Then
\[
  (g,t)\longmapsto \rho(g)\mathbf a_t
\]
is a homotopy from
\[
  p_q\circ\rho(g)=\rho(g)\mathbf e
\]
to $F_{\mathbf w}(g)$.  Since $F_{\mathbf w}$ is null-homotopic, and a constant
frame can be joined to $\mathbf e$ inside $V_{N,q}$, we obtain a null-homotopy
\[
  H\colon G\times I\longrightarrow V_{N,q}
\]
with
\[
  H(g,0)=p_q(\rho(g)),\qquad H(g,1)=\mathbf e.
\]

Lift $H$ through \eqref{eq:stiefel-fibration}, using $\rho$ as the prescribed lift
at time $0$.  The homotopy lifting property gives
\[
  \widetilde H\colon G\times I\longrightarrow\SO(N),
  \qquad
  \widetilde H(g,0)=\rho(g),
  \qquad
  p_q\circ\widetilde H=H.
\]
At time $1$,
\[
  p_q(\widetilde H(g,1))=\mathbf e,
\]
so $\widetilde H(g,1)$ lies in the fiber
$p_q^{-1}(\mathbf e)=\SO(N-q)$.  Thus $\rho$ is homotopic into
$\SO(N-q)$.
\end{proof}

\begin{remark}\label{rem:ordinary-homotopy}
The terminal map in \cref{thm:fixed-frame} need not be a group homomorphism, and
the homotopy need not pass through representations.  The conclusion concerns the
ordinary homotopy class of $\rho$.  This is precisely the level needed for
clutching over spheres.
\end{remark}

\section{Adjoint compression}\label{sec:adjoint}

Let $G$ be a $d$-dimensional compact connected Lie group of rank $r$ with Lie algebra $\mathfrak g$. Choose an $\Ad(G)$-invariant inner product on $\mathfrak g$.  Since $G$ is
connected, the adjoint action is orientation preserving, so
\[
  \Ad\colon G\longrightarrow\SO(\mathfrak g)\cong\SO(d).
\]

\begin{theorem}\label{thm:adjoint-compression}
The adjoint map is homotopic, as a map of spaces, to a map with values in
\[
  \SO(d-r+1)\subset\SO(d).
\]
Equivalently, it is $(r-1)$-compressible.
\end{theorem}

\begin{proof}
For $r=1$ there is nothing to prove.  Suppose $r\ge2$, and let $T\subset G$ be a
maximal torus with Lie algebra $\mathfrak t$.  Since $T$ is abelian, conjugation by
an element of $T$ restricts to the identity on $T$, and therefore $T$ fixes
$\mathfrak t$ pointwise under the adjoint action.  Choose an $(r-1)$-dimensional
subspace of $\mathfrak t$ and apply \cref{thm:fixed-frame} with
\[
  N=d,\qquad K=T,\qquad q=r-1.
\]
The dimension condition is an equality:
\[
  \dim(G/T)=d-r=d-(r-1)-1=N-q-1.
\]
\end{proof}

The loss of one fixed direction is exactly what is required by the connectivity
range: using all $r$ directions of $\mathfrak t$ would ask for the false inequality
$d-r\le d-r-1$.

\section{Adjoint bundles and tangent spheres}\label{sec:bundles-spheres}

We first recall the clutching consequence of \cref{thm:adjoint-compression}.  Let
$P\to S^m$ be a principal $G$-bundle, with clutching map
\[
  f\colon S^{m-1}\longrightarrow G.
\]
For any representation $\rho\colon G\to\SO(V)$, the associated vector bundle
$P\times_\rho V$ has clutching map $\rho\circ f$.  In particular,
\[
  \ad(P):=P\times_{\Ad}\mathfrak g
\]
has clutching map $\Ad\circ f$.

By \cref{thm:adjoint-compression}, there is a map
\[
  c\colon G\longrightarrow\SO(d-r+1)
\]
such that, after composing with the standard block inclusion
$j\colon\SO(d-r+1)\hookrightarrow\SO(d)$,
\[
  \Ad\simeq j\circ c.
\]
Precomposing the homotopy with $f$ gives the full clutching picture
\[
  S^{m-1}\xrightarrow{\ f\ }G\xrightarrow{\ \Ad\ }\SO(d)
  \qquad\simeq\qquad
  S^{m-1}\xrightarrow{\ f\ }G\xrightarrow{\ c\ }\SO(d-r+1)
  \xrightarrow{\ j\ }\SO(d).
\]
Homotopic clutching maps define isomorphic vector bundles.  Since $j$ fixes the
last $r-1$ coordinate vectors, those coordinates glue trivially.  We obtain the
following consequence.

\begin{corollary}\label{cor:adjoint-splitting}
For every principal $G$-bundle $P\to S^m$ there is an oriented real vector bundle
$\xi$ of rank $d-r+1$ such that
\[
  \ad(P)\cong \xi\oplus\eps^{r-1}.
\]
In particular, $\ad(P)$ has at least $r-1$ everywhere linearly independent
sections.
\end{corollary}

To apply this to tangent bundles, we use Adams's vector-field theorem \cite{Adams1962}; see also
\cite{Thomas1969}.

\begin{theorem}[Adams]\label{thm:adams}
Write a positive integer $m$ uniquely in the form
\[
  m=(2a+1)2^{4c+b},\qquad 0\le b\le3,
\]
and define the Radon--Hurwitz number by $\rho(m)=8c+2^b$.
Then
\[
  \Span(S^{m-1})=\rho(m)-1.
\]
\end{theorem}

\begin{theorem}\label{thm:rank-obstruction}
Let $G$ be a $d$-dimensional compact connected Lie group of rank $r$.  If
$TS^d$ admits a reduction through
\[
  \Ad\colon G\longrightarrow\SO(d),
\]
then
\[
  r\le\rho(d+1).
\]
\end{theorem}

\begin{proof}
An adjoint reduction means that for some principal $G$-bundle $P\to S^d$,
\[
  TS^d\cong P\times_{\Ad}\mathfrak g=\ad(P).
\]
By Corollary~\ref{cor:adjoint-splitting}, the tangent bundle splits off
$\eps^{r-1}$.  Hence $S^d$ has at least $r-1$ everywhere independent tangent
vector fields.  \Cref{thm:adams} gives
\[
  r-1\le\rho(d+1)-1.
\]
\end{proof}

For a first general consequence we use two elementary facts about compact
connected Lie groups.  First,
\begin{equation}\label{eq:rank-parity}
  \dim G-\rank G\quad\text{is even},
\end{equation}
because the nonzero roots occur in real two-dimensional root spaces  \cite[Chapter 5, Proposition 2.7]{BrockerTomDieck1985}.  Second, a
compact connected Lie group of rank one has dimension $1$ or $3$: its Lie algebra
is either one-dimensional abelian or has semisimple part of rank one, hence is
$\mathfrak{su}(2)$.  See, for example, \cite[Chapter 5, Corollary 1.6]{BrockerTomDieck1985}.

\begin{corollary}\label{cor:mod-four}
Let $G$ be a nontrivial compact connected Lie group of dimension $d>1$.  If $d\not\equiv3\pmod4$, 
then $TS^d$ admits no reduction through the adjoint representation of $G$.
\end{corollary}

\begin{proof}
Put $r=\rank G$.  If $d$ is even, then \eqref{eq:rank-parity} makes $r$ even, so
$r\ge2$, and $\rho(d+1)=1$.  If
$d\equiv1\pmod4$, then $r$ is odd.  Since $d>1$, the rank-one case is impossible,
so $r\ge3$, whereas $\rho(d+1)=2$.  Both cases contradict
\cref{thm:rank-obstruction}.
\end{proof}

In particular, for the compact group of type $E_7$,
\[
  \dim E_7=133\equiv1\pmod4,
\]
so $TS^{133}$ admits no reduction through its adjoint representation. Any reduction to a subgroup of the adjoint copy of $E_7$ would extend to a reduction to that adjoint
$E_7$ and is therefore impossible as well.

The same obstruction is naturally compatible with products.  Up to finite cover,
a compact connected Lie group has the form
\[
  T^z\times G_1\times\cdots\times G_s,
\]
where the $G_i$ are compact simply connected simple groups \cite[Chapter V, Theorem 8.1]{BrockerTomDieck1985}.  Writing
$d_i=\dim G_i$ and $r_i=\rank G_i$, \cref{thm:rank-obstruction} gives the necessary
condition
\[
  z+\sum_i r_i
  \le
  \rho\!\left(z+\sum_i d_i+1\right).
\]
Thus the criterion applies without irreducibility assumptions and is particularly
natural for nonsimple groups.

As a final consequence, we see that adjoint reducibility of the corresponding tangent sphere is not preserved under taking direct powers: even if $TS^{\dim{G}}$ admits a reduction through the adjoint representation of $G$, the tangent bundles $TS^{m\dim G}$ cannot admit reductions through the adjoint representation of $G^m$ for all $m$.

\begin{corollary}\label{cor:large-powers}
Let $G$ be a fixed nontrivial compact connected Lie group.  For all sufficiently
large $m$, the tangent bundle $TS^{m\dim G}$ does not admit a reduction through the
adjoint representation of $G^m$.
\end{corollary}

\begin{proof}
Write $d_0=\dim G$ and $r_0=\rank G$.  Such a reduction would require
\[
  mr_0\le\rho(md_0+1).
\]
The Radon--Hurwitz function satisfies
\[
  \rho(n)\le 2\log_2 n+2
\]
for every positive integer $n$.  Hence
\[
  \rho(md_0+1)\le 2\log_2(md_0+1)+2.
\]
The left-hand side of the required inequality grows linearly in $m$, whereas
the right-hand side grows only logarithmically.  The inequality therefore fails
for all sufficiently large $m$.
\end{proof}

\subsection{Relation with classical \texorpdfstring{$G$}{G}-structure results}

The preceding statements should not be read as novelty claims for every individual
simple group.  Leonard's work gives broad non-existence results for proper reductions
of the tangent structure group, including the real even-dimensional case apart from
a low-dimensional exception; Ozaki sharpened several results for proper subgroups of
the classical tangent groups \cite{Leonard1971,Ozaki1991}.  \v{C}adek and Crabb
later classified reductions through homomorphisms whose source and target are among
$\SO(k)$, $\SU(k)$ and $\Sp(k)$, and in particular ruled out the relevant
irreducible classical-source representations except for their listed exceptions
\cite{CadekCrabb2006}.

For bundles over suspensions, \v{C}adek and Crabb also prove a general
compression statement: if a reduction factors through
$\rho\colon G\to\SO(n)$ and $\dim G\le n-j-1$, then it further reduces to the
standard subgroup $\SO(n-j)$ \cite[Proposition~3.1]{CadekCrabb2006}.  The
fixed-frame argument of \cref{thm:fixed-frame} replaces $\dim G$ by the smaller
homogeneous-space dimension $\dim(G/K)$ whenever a subgroup $K$ fixes the relevant
frame pointwise.  This is the feature that makes the adjoint application effective
with a maximal torus.

Our focus is thus different: the source is an arbitrary compact connected Lie group, but
the homomorphism is specifically the adjoint map. The maximal torus then supplies fixed directions uniformly, so the Stiefel argument
reduces to the single numerical obstruction of \cref{thm:rank-obstruction}.  This
is weaker than the strongest classical classification theorems in many particular
families, but it applies without assuming that the source is classical or simple.
The exceptional $E_7$ case is especially useful because it is not covered by the
classical-source classification and is exactly the alternative left in the
$133$-dimensional argument of \cite{BorEtAl2021}.

\section{The \texorpdfstring{$133$}{133}-dimensional application}\label{sec:banach}

We briefly explain where the preceding obstruction enters the geometric argument of
Bor, Hern\'andez-Lamoneda, Jim\'enez-Desantiago and Montejano.  Their topological
classification contains the following statement.

\begin{theorem}[Bor et al.]
\label{thm:bor-structure}
Let $n\equiv1\pmod4$, $n\ge5$, and suppose that the structure group of $TS^n$
reduces to a closed connected subgroup $G\subset\SO(n)$.
\begin{enumerate}[label=\textup{(\alph*)},leftmargin=2.2em]
\item If $G$ is reducible, then it is conjugate to a subgroup of the standard
$\SO(n-1)\subset\SO(n)$ and acts transitively on $S^{n-2}$.
\item If $G$ is irreducible, then either $G=\SO(n)$, or $n=133$ and
$G\subset H\subset\SO(133)$, where $H$ is the adjoint group of type $E_7$.
\end{enumerate}
\end{theorem}

This is \cite[Theorem~1.6]{BorEtAl2021}.  The exceptional alternative in part (b)
is impossible by Corollary~\ref{cor:mod-four}.  Consequently, in dimension $133$ every
closed connected irreducible reduction occurring in \cref{thm:bor-structure} is the
full group $\SO(133)$.

The rest of the argument of \cite{BorEtAl2021} is geometric.  A family of mutually
linearly equivalent central hyperplane sections gives a reduction of the tangent
structure group to the symmetry group of a normalized model section.  Their
\cref{thm:bor-structure}, together with their characterization of ellipsoids by
affine bodies of revolution, then proves Banach's conjecture in the dimensions they
consider.  The only excluded branch at $n=133$ is the adjoint $E_7$ possibility.
Thus the same argument, with Corollary~\ref{cor:mod-four} inserted at that point, gives:

\begin{corollary}\label{cor:banach133}
Let $V$ be a real normed space with $\dim V>133$.  If all $133$-dimensional linear
subspaces of $V$ are mutually isometric, then $V$ is a Hilbert space.
\end{corollary}

The same $E_7$ exclusion also removes the corresponding $133$-dimensional
exception wherever the section and projection results in
\cite{BorEtAl2021,MontejanoProj2020,MontejanoECM2023} invoke the same
structure-group alternative.

\section{Further questions}\label{sec:further}

Adams's obstruction rules out broad families of adjoint reductions, but it leaves a natural classification problem in the remaining dimensions.

\begin{problem}\label{prob:classification}
Classify the compact connected Lie groups $G$ for which
\[
  TS^{\dim G}
\]
admits a reduction through the adjoint representation
$\Ad\colon G\to\SO(\dim G)$.
\end{problem}

\Cref{thm:rank-obstruction} gives a necessary condition
 and Corollary~\ref{cor:mod-four} settles all dimensions not congruent to $3$ modulo $4$,
apart from the trivial one-dimensional case.  The remaining range contains the
automatically positive parallelizable dimensions $3$ and $7$, but in higher
dimensions additional unstable obstructions are needed.  It would be interesting to
determine how far the rank obstruction, together with the existing classification
results for $G$-structures on spheres, comes toward a complete answer.

More generally, \cref{thm:fixed-frame} applies to any orthogonal representation for
which a subgroup fixes a sufficiently large frame and the quotient $G/K$ lies in the
connectivity range of the corresponding Stiefel manifold.  The adjoint
representation is a particularly natural instance because maximal tori provide the
fixed directions canonically. This suggests a second question.

\begin{problem}
Let $G$ be a nontrivial compact connected Lie group of dimension $d$ and rank $r$.
Determine the largest integer $0\le q<d$ for which
\[
  \Ad\colon G\to\SO(d)
\]
is homotopic to a map with values in $\SO(d-q)$.  In particular, characterize
when the maximal-torus bound $q\ge r-1$ supplied by
\cref{thm:adjoint-compression} is sharp.
\end{problem}

The maximal-torus bound need not be optimal: other subgroups may provide larger fixed frames, and further compression may also arise from unstable homotopy-theoretic phenomena not detected by the fixed-frame criterion.

\section*{Acknowledgments}
The author thanks Luis Montejano for general discussions during the development of this work. This work was supported by UNAM--PAPIIT grant
IN119026 and was developed with support from UNAM DGAPA--PASPA during a sabbatical stay at the Instituto de
Matem\'aticas, Unidad Juriquilla, UNAM.

The work was done in collaboration with OpenAI ChatGPT (GPT-5.6 Sol). The model formulated an initial sketch of the compression argument and assisted with the organization and exposition of the manuscript. The argument was subsequently carefully audited, developed, revised, and fully understood by the author. The author assumes full responsibility for the mathematical content and for the accuracy and completeness of the references and literature review.

\end{document}